\documentclass[12pt,reqno]{amsart}
\usepackage[all]{xy}
\usepackage{amssymb,amsmath,amsthm,mathrsfs}
\usepackage{cases}
\usepackage{enumerate}
\usepackage{hyperref}
\usepackage{bm}
\usepackage[firstinits=true, maxnames=99, maxalphanames=99]{biblatex}
\renewbibmacro{in:}{}
\usepackage{mathtools}
\usepackage{dsfont}
\usepackage{multirow}
\usepackage{float}

\allowdisplaybreaks[3]

\numberwithin{equation}{section}
\allowdisplaybreaks[3]
\makeatother
\newtheorem{thm}{Theorem}[section]
\newtheorem{Con}[thm]{Conjecture}

\newtheorem{cor}[thm]{Corollary}

\newtheorem{pro}[thm]{Proposition}

\newtheorem{lem}[thm]{Lemma}

\theoremstyle{definition}
\newtheorem{defin}[thm]{Definition}
\newtheorem{rem}[thm]{Remark}

\numberwithin{equation}{section}

\theoremstyle{definition}
\newtheorem*{ack}{Acknowledgements}

\newcommand{\Q}{\mathbb{Q}}

\newcommand{\R}{\mathbb{R}}

\newcommand{\eps}{\varepsilon}

\newcommand{\si}{\sigma}

\newcommand{\abs}[1]{\left\lvert#1\right\rvert}

\newcommand{\vv}{\mathbf{v}}
\newcommand{\Log}[1]{\operatorname{Log}(#1)}
\newcommand{\CycUnitGp}{\mathcal{C}}

\newcommand{\prob}[1]{\operatorname{Pr}\left[ #1\right]}

\newcommand{\bb}{\mathbf{b}}
\newcommand{\norm}[1]{\left\lVert#1\right\rVert}
\newcommand{\Ok}{\mathcal{O}_K}
\newcommand{\bigO}[1]{O\left(#1\right)}

\begin{document}
\title[Mean square of inverses of Dirichlet $L$-functions involving conductors]{Mean square of inverses of Dirichlet $L$-functions involving conductors} 
\author[I.-I. Ng, Y. Toma]{Iu-Iong Ng and Yuichiro Toma}
\address[Ng]{Graduate School of Mathematics, Nagoya University,  Furo-cho, Chikusa-ku, Nagoya 464-8602, Japan.}
\email{m20048d@math.nagoya-u.ac.jp}
\address[Toma]{Graduate School of Mathematics, Nagoya University,  Furo-cho, Chikusa-ku, Nagoya 464-8602, Japan.}
\email{m20034y@math.nagoya-u.ac.jp}
\makeatletter
\@namedef{subjclassname@2020}{\textup{2020} Mathematics Subject Classification}
\makeatother
\maketitle

\begin{abstract}
We deal with negative square moments of Dirichlet $L$-functions. Summing over characters modulo $q$, we obtain an asymptotic formula for the negative second moment of $L(1,\chi)$ involving conductors. As an application, we give the improved lower bound on the success probability of the algorithm which recovers a short generator of the input generator of a principal ideal sampled from a specific Gaussian distribution in cyclotomic number fields.
\end{abstract} 

\section{Introduction}
Let $s=\sigma+it$ be a complex variable. In analytic number theory, special values of Dirichlet $L$-functions on the real axis have been receiving considerable attention, such as vanishing or non-vanishing at the central point $s=1/2$, and the quantities $L(1,\chi)$ for estimating the class number of cyclotomic fields. Let $q$ be a positive integer and let $\chi$ be a Dirichlet character modulo $q$. The corresponding Dirichlet $L$-function is defined to be 
\[
L(s,\chi):=\sum_{n=1}^\infty \frac{\chi(n)}{n^s}
\]
for $\si>1$, and can be continued analytically over the whole plane, except for the possible pole at $s=1$. 

The study of the mean values of Dirichlet $L$-functions at $s=1$ can be traced back to Paley~\cite{Paley32} and Selberg~\cite{Sel46} in order to estimate the class number of the cyclotomic field $\Q (\xi_q)$, where $\xi_q$ is a primitive $q$-th root of unity. They proved in the case of $q$ being a prime number that
\begin{align*}
    \sum_{\chi \neq \chi_0} \abs{L(1,\chi)}^2 = \zeta(2)q+O\left((\log q)^2\right).
\end{align*}
The above formula was further studied by Slavutskii~\cite{S85, S86}, Zhang~\cite{Z90-1}, and by Katsurada and Matsumoto~\cite{KM94}. For more general $q$
, the asymptotic formulas for the $2k$-th power mean value of $\abs{L(1,\chi)}$ were proved by Zhang and Wang~\cite{ZW}. 

\subsection{Negative moment}
\label{sec:negative}
Compared with (positive) mean values, negative mean values of Dirichlet $L$-functions at $s=1$ have not been much studied. Zhang~\cite{Z93} first studied the $2k$-th negative moments of Dirichlet $L$-functions at $s=1$. Later, Zhang and Deng~\cite{ZD02} also considered a similar sum. For the family of quadratic Dirichlet $L$-functions, Granville and Soundararajan~\cite{GS03} obtained asymptotic formulae for negative moments of $L(1, \chi_d)$. Further, in the function field setting, negative moments of quadratic Dirichlet $L$-functions were obtained by Lumley~\cite{Lu19}, and shifted negative moments of quadratic Dirichlet $L$-functions over function fields were proved by Bui, Florea and Keating~\cite{BFK} and Florea~\cite{F24}. In addition, Ihara, Murty and Shimura~\cite{IMS09} and Matsumoto and Saad Eddin ~\cite{MS} studied the mean value of the $2k$-th power of the logarithmic derivatives of Dirichlet $L$-functions at $s=1$. 

The asymptotic relation for negative moments of the Riemann zeta-function $\zeta(s)$ was considered by Gonek~\cite{G89}, and the following conjecture was proposed.

\begin{Con}[Gonek]
    Let $k > 0$ be fixed. Then 
    \begin{align*}
        \frac{1}{T} \int_0^T \abs{\zeta\left( \frac{1}{2}+\frac{\delta}{\log T}+it\right)}^{-2k} dt \asymp \left( \frac{\log T}{\delta}\right)^{k^2}
    \end{align*}
    holds uniformly for $1 \leq \delta \leq \log T$, and 
    \begin{align*}
        \frac{1}{T} \int_0^T \abs{\zeta\left( \frac{1}{2}+\frac{\delta}{\log T}+it\right)}^{-2k} dt \asymp \begin{cases}
            (\log T)^{k^2} & k <\frac{1}{2}, \\
            (\log \frac{e}{\delta})(\log T)^{k^2} & k =\frac{1}{2}, \\
            \delta^{1-2k} (\log T)^{k^2} & k >\frac{1}{2}
        \end{cases}
    \end{align*}
    holds uniformly for $0 < \delta \leq 1$.
\end{Con}
Gonek~\cite{G89} proved lower bounds which attain the conjectural order of magnitude for $1 \leq \delta \leq \log T$
and all $k > 0$, and for $0 < \delta \leq 1$ for $k < 1/2$ 
assuming the Riemann Hypothesis (RH). Florea and Bui~\cite{BF24} obtained upper bounds in some ranges of $\delta$ under the RH. However, for the case $k > 1$ and $0<\delta \leq 1$, the above conjecture seems to contradict with random matrix theory computations due to Berry and Keating~\cite{BK02}, Fyodorov and Keating~\cite{FK03}, and Forrester and Keating~\cite{FK04}. For more details, see pp.248 in \cite{BF24}.

\subsection{Main results}
In~\cite{NT24+}, the authors proved an asymptotic formula for 
\begin{align*}
    \sum_{\chi \neq \chi_0}\abs{L(1,\chi)}^{-2k}, 
\end{align*}
where $k \in \mathbb{N}$, assuming the truth of the Generalized Riemann Hypothesis (GRH). In this paper, by the standard method that is also employed in~\cite{MS}, we remove the assumption of the GRH in the result given in~\cite{NT24+} for $k=2$.

\begin{thm} 
    \label{thm:negative square moment}
    Let $q$ be a positive integer and $\chi$ be a Dirichlet character modulo $q$. Then we have
    \begin{align*}
    \sum_{\substack{\chi \neq \chi_0 \\ \chi(-1)=1}} \frac{1}{\abs{L(1,\chi)}^{2}}&= \frac{\zeta(2)}{2\zeta(4)}\prod_{p \mid q} \left( 1+\frac{1}{p^2}\right)^{-1} \varphi(q) + O\left(\exp \left( C\frac{\log q}{\log\log q}\right) \right) \\
    &\quad+\delta_1\cdot O\left( \left(1-\beta_1\right)^{-1}\left( (\log q)^2+ (1-\beta_1)^{-1}\right)\right)
    \end{align*}
    for an absolute constant $C>0$, where $\beta_1$ denotes the exceptional zero (defined in Proposition \ref{pro:exceptional zero}), and $\delta_1=1$ if $\beta_1$ exists, or $\delta_1=0$ otherwise.
\end{thm}
From Siegel’s theorem (see \cite[Corollary 11.15]{MV}) which asserts that $1-\beta_1 \geq C(\eps)q^{-\eps}$, we have $\left(1-\beta_1\right)^{-1}\left( (\log q)^2+ (1-\beta_1)^{-1}\right) \ll_\eps  q^{\eps}$ by resetting $\eps$. Hence the above becomes 
\begin{align}
\label{consequence of the Siegel Theorem}
    \sum_{\substack{\chi \neq \chi_0 \\ \chi(-1)=1}} \frac{1}{\abs{L(1,\chi)}^{2}}&=\frac{\zeta(2)}{2\zeta(4)}\prod_{p \mid q} \left( 1+\frac{1}{p^2}\right)^{-1} \varphi(q) +O_{\eps}\left(q^\eps\right).
\end{align}

Moreover, we give the asymptotic formula for the negative square moment involving the conductor under the assumption of the nonexistence of the exceptional zero for prime power $q$. 

\begin{thm} 
    \label{thm:negative square moment with conductor}
    Let $q=p^k$ be a power of a prime and $\chi$ be a Dirichlet character modulo $q$. Let $f_\chi$ be the conductor of $\chi$. Assuming that the exceptional zero does not exist, then we have
    \begin{align*}
    \sum_{\substack{\chi \neq \chi_0 \\ \chi(-1)=1}} \frac{1}{f_\chi\abs{L(1,\chi)}^{2}}&=\frac{\zeta(2)}{2\zeta(4)} \frac{(p-1)^2}{p^2+1} k + O\left(\frac{1}{\log p}\right) \\
    & \qquad +O\left(\frac{k^2(\log p)^2(\log k+\log\log p)^2}{p}\right)
    \end{align*}
     for $k =o\left(\frac{p}{(\log p)^4}\right)$; otherwise, the above formula implies only the upper bound estimate, i.e.,
     \[
\sum_{\substack{\chi \neq \chi_0 \\ \chi(-1)=1}} \frac{1}{f_\chi\abs{L(1,\chi)}^{2}} \ll 
    \frac{k^2(\log p)^2(\log k+\log\log p)^2}{p}. 
\]
\end{thm}

If $q=p^k$ is a power of an odd prime $p$, then there is exactly one quadratic character with conductor $p$. If $q=2^k$, then there are at most three primitive quadratic characters with modulus $4,8$ (see \cite[Section 9.3]{MV}). Also, by using the lower bounds given by Landau~\cite{Lan}, $\abs{L(1,\chi)} \gg 1/ \log f_\chi$ for non-quadratic primitive character $\chi$ and $\abs{L(1,\chi)} \gg 1/ \sqrt{ f_\chi}$ for primitive quadratic character $\chi$, and the fact $\sum_{\substack{\chi \neq \chi_0 \\ \chi(-1)=1}} \frac{1}{f_\chi} \leq \frac{k}{2}$ (see \cite[Claim 3.5]{CDPR15}), one can only deduce that 
    \begin{align*}
    \sum_{\substack{\chi \neq \chi_0 \\ \chi(-1)=1}} \frac{1}{f_\chi\abs{L(1,\chi)}^{2}}&=\sum_{\substack{\chi \text{ : non-quadratic} \\ \chi(-1)=1}} \frac{1}{f_\chi\abs{L(1,\chi)}^{2}} + O\left(1\right) \\
     &\ll k^3 (\log p)^2.
    \end{align*}
Hence our Theorem \ref{thm:negative square moment with conductor} implies a nontrivial upper bound even for $k \gg \frac{p}{(\log p)^4}$ under the assumption of the nonexistence of the exceptional zero.

\begin{rem}
It is possible to obtain the negative square moment of $L(1,\chi)$ for an odd character $\chi$. By the same argument as in the proof of Theorem \ref{thm:negative square moment} and Theorem \ref{thm:negative square moment with conductor}, we can find that for any integer $q$,
\begin{align*}
\sum_{\substack{\chi \\ \chi(-1)=-1}} \frac{1}{\abs{L(1,\chi)}^2} &\sim \frac{\zeta(2)}{2\zeta(4)}\prod_{p \mid q} \left( 1+\frac{1}{p^2}\right)^{-1} \varphi(q), 
\end{align*}
and for $q=p^k$ with $k=o\left( \frac{p}{(\log p)^4}\right)$ that
\begin{align*}
\sum_{\substack{\chi \\ \chi(-1)=-1}} \frac{1}{f_\chi \abs{L(1,\chi)}^{2}} &\sim \frac{\zeta(2)}{2\zeta(4)} \frac{(p-1)^2}{p^2+1} k.
\end{align*}
\end{rem}

\section{Auxiliary lemmas}
In order to obtain Theorem \ref{thm:negative square moment} and Theorem \ref{thm:negative square moment with conductor}, we prove auxiliary lemmas in this section. First, we recall some well-known results from~\cite{MV}. 

\begin{pro}
\label{pro:exceptional zero}
    Let $q \geq 1$. There is an effectively computable positive constant $c_0$ such that 
    \[
    \prod_{\chi \bmod q} L(s,\chi)
    \]
    has at most one zero $\beta_1$ in the region
    \[
    \si \geq 1-\frac{c_0}{\log q(\abs{t}+1)}.
    \]
    Such a zero, if it exists, is real simple and corresponds to a nonprincipal real character which we denote by $\chi_1$.
\end{pro}
\begin{proof}
    See \cite[Theorem 11.3]{MV}.
\end{proof}

\begin{pro}
    Let $\chi$ be a nonprincipal character modulo $q$ and suppose that $\si>c_0/(2 \log q(\abs{t}+1))$. If $L(s,\chi)$ has no exceptional zero , or if $\beta_1$ is an exceptional zero of $L(s,\chi)$ but $\abs{s-\beta_1} \geq 1/ \log q$, then 
    \begin{align}
        \label{L-inverse estimate}
        \frac{1}{L(s,\chi)} \ll \log q(\abs{t}+1).
    \end{align}
    Alternatively, if $\beta_1$ is an exceptional zero of $L(s,\chi)$ and $\abs{s-\beta_1} \leq 1/ \log q$, then 
    \begin{align}
        \label{L estimate near beta}
        \abs{s-\beta_1} \ll \abs{L(s,\chi)} \ll \abs{s-\beta_1} (\log q)^2.
    \end{align}
\end{pro}
\begin{proof}
    See \cite[Theorem 11.4]{MV}.
\end{proof}

\begin{lem}
    Let $\beta_1$ be the exceptional zero corresponding to $\chi_1$. Then, the Laurent expansion of the function $1/L(s,\chi_1)$ at the point $s = \beta_1$,
    \begin{align*}
    \frac{1}{L(s,\chi_1)} &= \sum_{n=-1}^\infty P_n (s-\beta_1)^n,
\end{align*}
satisfies
\begin{align}
\label{Pn estimate}
    P_n &= O\left((\log q)^{n+1}\right).
\end{align}
\end{lem}
\begin{proof}
Since $\beta_1$ is a simple pole of $1/L(s,\chi_1)$, we have
\begin{align*}
    \frac{1}{L(s,\chi_1)} &= \sum_{n=-1}^\infty P_n (s-\beta_1)^n.
\end{align*}
For $n \geq 0$, $P_n$ is given by
\begin{align*}
    \frac{1}{2\pi i} \int_{\mathcal{C}} \frac{ds}{L(s,\chi_1)(s-\beta_1)^{n+1}}.
\end{align*}
Here the contour $\mathcal{C}$ is a positively oriented circle of radius $R = c_2/ \log(2q)$ and centered at 
$\beta_1$, where $c_2 < c_0/2$ is sufficiently small. By using (\ref{L estimate near beta}), we get 
\begin{align*}
    P_n & \leq  \frac{1}{2\pi } \int_{0}^{2\pi} \frac{Rd \theta}{\abs{Re^{i \theta}} \abs{Re^{i\theta}}^{n+1}} \leq (\log q)^{n+1}.
\end{align*}
Otherwise for $n=-1$, we have $P_{-1}= 1/L^\prime(\beta_1, \chi_1)$. Since 
\begin{align*}
    \frac{1}{L^\prime (s,\chi_1)} &= \lim_{s \to \beta_1} \frac{s-\beta_1}{L(s,\chi_1)},
\end{align*}
by using (\ref{L estimate near beta}) we obtain $P_{-1} \ll 1$. Hence (\ref{Pn estimate}) is also valid for the case $n=-1$.
\end{proof}

\begin{lem}[Character orthogonality]
\label{lem:Character orthogonality}
     Let $\mathfrak{a}, \mathfrak{b} \in \{0,1\}$. For $n_1, n_2$ integers coprime to $q$, we have
     \begin{align*}
     &\sum_{\substack{\chi \neq \chi_0 \\ \chi(-1) =(-1)^{\mathfrak{a}}}} \frac{\chi(n_1)\overline{\chi}(n_2)}{{f_\chi}^{\mathfrak{b}}} \\
     &=\frac{1}{2} \sum_{\substack{d\mid q \\ d>1}}\frac{1}{d^{\mathfrak{b}}} \left( \sum_{\substack{l \mid d \\ n_1 \equiv n_2 \pmod l}}\varphi(l)\mu(d/l)+(-1)^{\mathfrak{a}} \sum_{\substack{l \mid d \\ n_1 \equiv -n_2 \pmod l}}\varphi(l)\mu(d/l) \right).
     \end{align*}
     The sum on the left-hand side vanishes if $(n_1n_2, q)\neq 1$.
\end{lem}
\begin{proof}
This lemma can be proved in a standard way similar to Lemma 1 in \cite{So}. Let $A_{q,d}$ be the set of all Dirichlet characters modulo $q$ whose conductor is $d$, and $B_{d}$ be the set of all primitive Dirichlet characters with conductor $d$. For $n_1, n_2$ integers coprime to $q$, it can be rewritten as 
    \begin{align}
    \label{Character sum}
    \sum_{\chi \neq \chi_0} \frac{\chi(n_1)\overline{\chi}(n_2)}{{f_\chi}^{\mathfrak{b}}} = \sum_{\substack{d\mid q \\ d>1}} \frac{1}{d^{\mathfrak{b}}} \sum_{\chi \in A_{q,d}}\chi(n_1)\overline{\chi}(n_2)=\sum_{\substack{d\mid q \\ d>1}} \frac{1}{d^{\mathfrak{b}}} \sum_{\chi \in B_{d}}\chi(n_1)\overline{\chi}(n_2).
    \end{align}
Let $h_{n_1,n_2}(l) := \sum_{\chi \in B_{l}} \chi(n_1)\overline{\chi}(n_2)$. Since
    \[
    \sum_{l \mid d} h_{n_1,n_2}(l) = \sum_{\chi \pmod d} \chi(n_1)\overline{\chi}(n_2) = \begin{cases}
        \varphi(d) & n_1 \equiv n_2 \pmod d, \\
        0 & otherwise,
    \end{cases} 
    \]
M\"obius inversion implies that
    \[
    \sum_{\chi \in B_d} \chi(n_1)\overline{\chi}(n_2) = \sum_{\substack{l \mid d\\n_1 \equiv n_2 \pmod l}} \varphi(l)\mu(d/l).
    \]
Hence combining (\ref{Character sum}), we have
    \begin{align*}
    \sum_{\chi \neq \chi_0} \frac{\chi(n_1)\overline{\chi}(n_2)}{{f_\chi}^{\mathfrak{b}}} = \sum_{\substack{d\mid q \\ d>1}} \frac{1}{d^{\mathfrak{b}}} \sum_{\substack{l \mid d\\n_1 \equiv n_2 \pmod l}} \varphi(l)\mu(d/l).
    \end{align*}
Replacing $n_2$ with $-n_2$, we get the desired result.
    \end{proof}

\begin{lem}
\label{lemma weighted sum}
    Let $m, n$ and $q$ be positive integers. For any $l \in \mathbb{N}$ and $X > 1$, we have
    \begin{align*}
        &\sum_{\substack{m,n=1 \\ (mn,q)=1 \\ m\equiv \pm n \pmod l}} \frac{\mu(m)\mu(n)}{mn}e^{-\frac{X}{mn}} &= \frac{\zeta(2)}{\zeta(4)}\prod_{p \mid q} \left( 1+\frac{1}{p^2}\right)^{-1}+O\left( X^{-\frac{1}{2}}\right) + O\left( \frac{(\log X)^2}{l}\right).
    \end{align*}
\end{lem}
\begin{proof}
First, we calculate the diagonal contribution:
\begin{align*}
    \sum_{\substack{m=1 \\ (m,q)=1}}^\infty \frac{\mu(m)^2}{m^2}e^{-\frac{m^2}{X}} &= \sum_{\substack{m \leq X^{\frac{1}{2}} \\ (m,q)=1}} \frac{\mu(m)^2}{m^2}e^{-\frac{m^2}{X}} +\sum_{\substack{m>X^{\frac{1}{2}} \\ (m,q)=1}} \frac{\mu(m)^2}{m^2}e^{-\frac{m^2}{X}}.
\end{align*}
If $m >X^{\frac{1}{2}}$, then $e^{-\frac{m^2}{X}} \leq 1$. Therefore, the second sum on the right hand side in the above can be estimated as $X^{-\frac{1}{2}}$. Otherwise, by using $e^{-\frac{m^2}{X}}= 1+O\left( m^2/X\right)$, we find that the first sum on the right hand side is
\begin{align}
\label{diagonal contribution}
\sum_{\substack{m=1 \\ (m,q)=1}}^\infty \frac{\mu(m)^2}{m^2}+O\left( X^{-\frac{1}{2}}\right) = \frac{\zeta(2)}{\zeta(4)}\prod_{p \mid q}\left( 1+\frac{1}{p^2}\right)^{-1}+O\left( X^{-\frac{1}{2}}\right).
\end{align}

By the same manner as in~\cite{MS}, we can calculate the contribution which comes from $m \neq n$ and from $m \equiv n \pmod l$. In fact, we have
\begin{align}
\label{non-diagonal contribution 1}
    \sum_{\substack{m,n=1 \\ (mn,q)=1 \\ m \neq n \\ m \equiv n \pmod l}}^\infty \frac{\mu(m)\mu(n)}{mn}e^{-\frac{mn}{X}} &\ll \frac{(\log X)^2}{l}.
\end{align}

The point different from \cite{MS} is that we further treat the nondiagonal terms with $m \neq n$ and $m \equiv -n \pmod l$. 
\begin{align*}
    \sum_{\substack{m,n=1 \\ (mn,q)=1 \\ m \neq n \\ m \equiv -n \pmod l}}^\infty \frac{\mu(m)\mu(n)}{mn}e^{-\frac{mn}{X}} &\ll \sum_{\substack{m<n \\ (mn,q)=1 \\ m \equiv -n \pmod l}} \frac{e^{-\frac{mn}{X}}}{mn} \\
    &= \sum_{m=1}^\infty \frac{e^\frac{m^2}{X}}{m} \sum_{h>\frac{2m}{l}} \frac{e^{-\frac{mhl}{X}}}{-m+hl} \\
    &\ll \frac{1}{l}\sum_{m=1}^\infty \frac{e^{\frac{m^2}{X}}}{m} \int_\frac{2m}{l}^\infty \frac{e^{-\frac{mtl}{X}}}{t-\frac{m}{l}} dt \\
    &= \frac{1}{l}\sum_{m=1}^\infty \frac{e^{\frac{m^2}{X}}}{m} \left(\int_\frac{2m}{l}^\frac{X}{ml}+\int_\frac{X}{ml}^\infty \right) \frac{e^{-\frac{mtl}{X}}}{t-\frac{m}{l}} dt.
\end{align*}
We remark that if $\frac{2m}{l}<\frac{X}{ml}$, then $m \leq \sqrt{\frac{X}{2}}$. Hence the above is equivalent to  
\begin{align*}
    \frac{1}{l}\sum_{m\leq \sqrt{\frac{X}{2}}} \frac{e^{\frac{m^2}{X}}}{m} \left(\int_\frac{2m}{l}^\frac{X}{ml}+\int_\frac{X}{ml}^\infty \right) \frac{e^{-\frac{mtl}{X}}}{t-\frac{m}{l}} dt+\frac{1}{l}\sum_{m>\sqrt{\frac{X}{2}}} \frac{e^{\frac{m^2}{X}}}{m} \int_\frac{2m}{l}^\infty \frac{e^{-\frac{mtl}{X}}}{t-\frac{m}{l}} dt.
\end{align*}
Since $t-\frac{m}{l} \geq \frac{m}{l}$ for $t \geq \frac{2m}{l}$, the last term is
\begin{align*}
    \frac{1}{l}\sum_{m>\sqrt{\frac{X}{2}}} \frac{e^{\frac{m^2}{X}}}{m} \int_\frac{2m}{l}^\infty \frac{e^{-\frac{mtl}{X}}}{t+\frac{m}{l}} dt &\ll \sum_{m>\sqrt{\frac{X}{2}}} \frac{e^{\frac{m^2}{X}}}{m^2} \int_\frac{2m}{l}^\infty e^{-\frac{mtl}{X}} dt \\
    &\ll \frac{X}{l} \sum_{m>\sqrt{\frac{X}{2}}} \frac{e^{-\frac{m^2}{X}}}{m^3} \\
    &\ll \frac{1}{l}.
\end{align*}
By using the fact $e^{-mtl/X}\leq e^{-2m^2/X}$ for $t \geq \frac{2m}{l}$ and replacing $\frac{mtl}{X}=v$, we have
\begin{align*}
    &\frac{1}{l}\sum_{m\leq \sqrt{\frac{X}{2}}} \frac{e^{\frac{m^2}{X}}}{m} \left(\int_\frac{2m}{l}^\frac{X}{ml}+\int_\frac{X}{ml}^\infty \right) \frac{e^{-\frac{mtl}{X}}}{t-\frac{m}{l}} dt \\
    &\ll \frac{1}{l}\sum_{m\leq \sqrt{\frac{X}{2}}} \frac{e^{\frac{m^2}{X}}}{m} \int_\frac{2m}{l}^\frac{X}{ml}\frac{e^{-\frac{2m^2}{X}}}{t-\frac{m}{l}} dt+\frac{1}{l}\sum_{m\leq \sqrt{\frac{X}{2}}} \frac{e^{\frac{m^2}{X}}}{m} \int_1^\infty \frac{e^{-v}}{v-\frac{m^2}{X}} dv \\
    &\ll \frac{1}{l}\sum_{m\leq \sqrt{\frac{X}{2}}} \frac{e^{-\frac{m^2}{X}}}{m} \int_\frac{2m}{l}^\frac{X}{ml}\frac{dt}{t-\frac{m}{l}} +\frac{1}{l}\sum_{m\leq \sqrt{\frac{X}{2}}} \frac{e^{\frac{m^2}{X}}}{m} \int_1^\infty \frac{e^{-v}}{v-\frac{1}{2}} dv \\
    &\ll \frac{1}{l}\sum_{m\leq \sqrt{\frac{X}{2}}} \frac{e^{-\frac{m^2}{X}}}{m} \log \left( \frac{X}{m}+m\right) + \frac{1}{l}\sum_{m\leq \sqrt{\frac{X}{2}}} \frac{e^{\frac{m^2}{X}}}{m} \\
    &\ll \frac{(\log X)^2}{l}.
\end{align*}

Therefore, we have
\begin{align}
\label{non-diagonal contribution 2}
\sum_{\substack{m,n=1 \\ (mn,q)=1 \\ m \neq n \\ m \equiv -n \pmod l}}^\infty \frac{\mu(m)\mu(n)}{mn}e^{-\frac{mn}{X}} &\ll \frac{(\log X)^2}{l}. 
\end{align}
Combining (\ref{diagonal contribution}), (\ref{non-diagonal contribution 1}) and (\ref{non-diagonal contribution 2}), we obtain the desired result.
\end{proof}

\begin{lem}
\label{lem:S(X)}
    Let $\mathfrak{b} \in \{0,1 \}$. Let
\begin{align*}
    G_\mathfrak{b}(s):= \sum_{\substack{\chi\neq \chi_0 \\ \chi(-1)=1}} \frac{1}{{f_\chi}^{\mathfrak{b}} L(s,\chi)L(s,\overline{\chi})},
\end{align*}
and 
\begin{align}
\label{S(X)}
    S_\mathfrak{b}(X):= \frac{1}{2\pi i}\int_{3-i \infty}^{3+i \infty} G_\mathfrak{b}(s)X^{s-1}\Gamma(s-1)ds.
\end{align}
Then we have
\begin{align*}
    S_\mathfrak{b}(X) = &\frac{\zeta(2)}{2\zeta(4)}\prod_{p \mid q} \left( 1+\frac{1}{p^2}\right)^{-1} \sum_{\substack{d \mid q \\ d>1}} \frac{\varphi^*(d)}{d^{\mathfrak{b}}} \\
    &\quad+O\left( X^{-\frac{1}{2}} \sum_{\substack{d \mid q \\ d>1}} \frac{1}{d^{\mathfrak{b}}}\sum_{l \mid d} \varphi(l) \right)+O\left((\log X)^2 \sum_{\substack{d \mid q \\ d>1}} \frac{1}{d^{\mathfrak{b}}}\sum_{l \mid d} \frac{\varphi(l)}{l}\right),
\end{align*}
where $\varphi^*(n)$ denotes the number of primitive characters modulo $n$.
\end{lem}

\begin{proof}
By the well-known formula $e^{-y}=\frac{1}{2\pi i} \int_{2-i \infty}^{2+i \infty} y^s \Gamma(s)ds$, we have
\begin{align*}
    S_\mathfrak{b}(X) = \sum_{\substack{\chi\neq \chi_0 \\ \chi(-1)=1}} \frac{1}{{f_\chi}^{\mathfrak{b}}}\sum_{m=1}^\infty \sum_{n=1}^\infty \frac{\mu(m)\mu(n)\chi(m)\overline{\chi}(n)}{mn}e^{-\frac{X}{mn}}.
\end{align*}
Now we apply Lemma \ref{lem:Character orthogonality} with $\mathfrak{a}=0$, we have 
\begin{align*}
    S_\mathfrak{b}(X) = \frac{1}{2}\sum_{\substack{d \mid q \\ d>1}} \frac{1}{d^{\mathfrak{b}}}\sum_{l \mid d} \varphi(l)\mu(d/l)\sum_{\substack{m,n=1 \\ (mn,q)=1 \\ m=n \\ m\equiv \pm n \pmod l}}^\infty \sum_{n=1}^\infty \frac{\mu(m)\mu(n)}{mn}e^{-\frac{X}{mn}}.
\end{align*}
By applying Lemma \ref{lemma weighted sum}, we complete the proof.
\end{proof}

\section{Proof of Theorem \ref{thm:negative square moment}}
In this section, we prove Theorem \ref{thm:negative square moment}. We assume that $q$ is a positive integer throughout this section. We put $\mathfrak{b}=0$ in Lemma \ref{lem:S(X)} to obtain 
    \begin{align}
    \begin{split}
    \label{S_0(X)}
        S_0(X) &=\frac{\zeta(2)}{2\zeta(4)}\prod_{p \mid q} \left( 1+\frac{1}{p^2}\right)^{-1} \sum_{\substack{d \mid q \\ d>1}} \varphi^*(d) \\
    &\quad+O\left( X^{-\frac{1}{2}} \sum_{\substack{d \mid q \\ d>1}} \sum_{l \mid d} \varphi(l) \right)+O\left((\log X)^2 \sum_{\substack{d \mid q \\ d>1}} \sum_{l \mid d} \frac{\varphi(l)}{l}\right).
    \end{split}
    \end{align}

Now we estimate the integral in (\ref{S(X)}). Put
\[
g_\chi(s):= \frac{X^{s-1}\Gamma(s-1)}{L(s,\chi)L(s,\overline{\chi})}.
\]

Put $A(c_1) = 1 - c_1/ \log(q(T+1))$ with $0<c_1<c_0/2$ and shift the part $\abs{t} \leq T$ of the path of integration to the line segment $\si +it$ defined with $\si = A(c_1)$ and $\abs{t} \leq T$. Let $\mathcal{C}_T$ denote the closed contour that consists line segments joining the points $3-iT, 3+iT, A(c_1)+iT$ and $A(c_1)-iT$. If there is no exceptional zero, then we take $c_1=c_0/5$, where $c_0$ is the constant in Proposition \ref{pro:exceptional zero}. If there is an exceptional zero $\beta_1$ satisfying $\beta_1 < 1-c_0/(4\log (q(T+2)))$,  then we take $c_1=c_0/5$ as before. While, if there is an exceptional zero $\beta_1$ but $\beta_1 \geq 1-c_0/(4\log (q(T+2)))$, we take $c_1=c_0/3$. In the last case, $\mathcal{C}_T$ contains a simple pole at $s=1$ and a pole at $s=\beta_1$ of order $2$.  

If the exceptional zero $\beta_1$ with the associated character $\chi_1$ exists, then putting $u=s-\beta_1$, we write
\begin{align*}
    X^{s-1} = X^{\beta_1-1}\sum_{n=0}^\infty \frac{(\log X)^n}{n!} u^n
\end{align*}
and
\begin{align*}
    \Gamma(s-1) = \sum_{n=0}^\infty \frac{\Gamma^{(n)}(\beta_1-1)}{n!} u^n.
\end{align*}

Thus, $g_{\chi_1}(s)$ has 
\begin{align*}
    \operatorname{Res}( g_{\chi_1}(s) ; \beta_1) &= \lim_{s \to \beta_1} \frac{d}{ds} \left[ (s-\beta_1)^2 \frac{X^{s-1}\Gamma(s-1)}{L(s,\chi_1)^2}\right] \\
    &= X^{\beta_1-1} \lim_{u \to 0} \frac{d}{du} \left[ u^2 \frac{X^{u}\Gamma(u+\beta_1-1)}{L(u+\beta_1,\chi_1)^2}\right] \\
    &= X^{\beta_1-1} \lim_{u \to 0} \frac{d}{du} \left[ \sum_{n=0}^\infty R_n u^n\right] \\
    &= X^{\beta_1-1} R_1,
\end{align*}
where
\begin{align*}
    R_n &= \sum_{\substack{n_1,n_2,n_3 \geq 0 \\ n=n_1+n_2+n_3}} \left( \sum_{n_1= l_1+l_2} P_{l_1-1}P_{l_2-1}\right) \frac{(\log X)^{n_2}}{n_2!} \frac{\Gamma^{(n_3)}(\beta_1-1)}{n_3!}.
\end{align*}
By using the fact $\Gamma^{(m)}(\beta_1-1) \ll (1-\beta_1)^{-m-1}$, we have
\begin{align}
\label{residue at siegel zero}
    \operatorname{Res}( g_{\chi_1}(s) ; \beta_1) &\ll X^{\beta_1-1} \left(1-\beta_1\right)^{-1}\left( \log q+ \log X+ (1-\beta_1)^{-1}\right).
\end{align}

We apply the residue theorem to obtain
\begin{align*} 
    S_0(X) &= \sum_{\substack{\chi\neq \chi_0 \\ \chi(-1)=1}} \frac{1}{\abs{L(1,\chi)}^2} + \operatorname{Res}( g_{\chi_1}(s) ; \beta_1)\\
    &\quad+ \frac{1}{2\pi i} \left(\int_{3+iT}^{A(c_1)+iT} + \int_{A(c_1)+iT}^{A(c_1)-iT}+\int_{A(c_1)-iT}^{3-iT}\right)G_0(s)X^{s-1}\Gamma(s-1)ds \\
    &\quad+\frac{1}{2\pi i} \int_{\substack{\sigma=3 \\\abs{t} >T}} G_0(s)X^{s-1}\Gamma(s-1)ds.
\end{align*}
By (\ref{L-inverse estimate}) and the Stirling formula (see \cite[Theorem C.1]{MV})
\begin{align}
    \label{Stirling}
    \Gamma(\si+it)= \sqrt{2 \pi}(1+\abs{t})^{\si-\frac{1}{2}}e^{-\frac{\pi\abs{t}}{2}}(1+O(\abs{t}^{-1})),
\end{align}

we find that
\begin{align*}
    & \frac{1}{2\pi i} \int_{\substack{\sigma=3 \\\abs{t} >T}} G_0(s)X^{s-1}\Gamma(s-1)ds \\
    & \ll \varphi(q)X^2 \int_T^\infty \abs{\Gamma(2+it)} dt \\
    &\ll \varphi(q)X^2(T+1)^\frac{3}{2}e^{-\frac{\pi}{2}T},
\end{align*}
\begin{align*}
    & \frac{1}{2\pi i}\int_{A(c_1)+iT}^{A(c_1)-iT}G_0(s)X^{s-1}\Gamma(s-1)ds \\ 
    &\ll \varphi(q) (\log (q(T+1)))^2 X^{A(c_1)-1} \int_{A(c_1)-iT}^{A(c_1)+iT}\abs{\Gamma(s-1)}\abs{ds} \\
    &\ll \varphi(q) (\log (q(T+1)))^2 X^{A(c_1)-1},
\end{align*}
and 
\begin{align*}
    & \frac{1}{2\pi i}\int_{A(c_1) \pm iT}^{3 \pm iT}G_0(s)X^{s-1}\Gamma(s-1)ds \\ 
    &\ll \varphi(q) (\log (q(T+1)))^2 X^{-1} (1+T)^{-\frac{3}{2}} e^{-\frac{\pi T}{2}}\int_{A(c_1)}^{3} ((1+T)X)^{\si}d\si \\
    &\ll \varphi(q) (\log (q(T+1)))^2 \frac{X^{2} (1+T)^{\frac{3}{2}} e^{-\frac{\pi T}{2}}}{\log ((1+T)X)}.
\end{align*}
We now put $T=q$ and $X =\exp\left(\frac{4}{c_0} (\log q)^2 \right)$. Then we have
\begin{align}
\begin{split}
    \label{integral-|t|>T}
    \frac{1}{2\pi i} \int_{\substack{\sigma=3 \\\abs{t} >T}} G_0(s)X^{s-1}\Gamma(s-1)ds \ll \varphi(q) \exp \left( -\frac{\pi}{2}q\left( 1+O \left(\frac{(\log q)^2}{q}\right)\right)\right), 
\end{split}
\end{align}
\begin{align}
\label{integral-v}
    \frac{1}{2\pi i}\int_{A(c_1)+iT}^{A(c_1)-iT}G_0(s)X^{s-1}\Gamma(s-1)ds \ll \frac{\varphi(q)(\log q)^2}{q^2}, 
\end{align}
and 
\begin{align}
\begin{split}
\label{integral-h}
    &\frac{1}{2\pi i}\int_{A(c_1) \pm iT}^{3 \pm iT}G_0(s)X^{s-1}\Gamma(s-1)ds \\
    &\ll \varphi(q)(\log q) \exp \left(-\frac{\pi}{2}q\left( 1+O\left( \frac{\log q)^2}{q}\right)\right)\right).
    \end{split}
\end{align}
Also from (\ref{residue at siegel zero}), we find that the contribution from the exceptional zero is
\begin{align*}
    \operatorname{Res}( g_{\chi_1}(s) ; \beta_1) &\ll  \left(1-\beta_1\right)^{-1}\left( (\log q)^2+ (1-\beta_1)^{-1}\right).
\end{align*}

Finally, from the fact $n = \sum_{d \mid n} \varphi(d)$, we have
\begin{align}
\begin{split}
    \label{error-1}
    X^{-\frac{1}{2}} \sum_{\substack{d \mid q \\ d>1}} \sum_{l \mid d} \varphi(l) &\ll \exp\left( -\frac{2}{c_0}(\log q)^2\right) \sum_{\substack{d \mid q \\ d>1}} d \\
    &\ll q(\log\log q) \exp\left( -\frac{2}{c_0}(\log q)^2\right).
\end{split}
\end{align}
Here we use Gronwall's theorem (see \cite[Theorem 323]{HW}) in the last step. Similarly, we use $\varphi(n) \leq n$ to obtain 
\begin{align*}
    (\log X)^2 \sum_{\substack{d \mid q \\ d>1}} \sum_{l \mid d} \frac{\varphi(l)}{l} &\ll (\log q)^4 \sum_{\substack{d \mid q \\ d>1}} \tau_2(q) \\
    &\ll (\log q)^4 \tau_3(q) \\
    & \ll (\log q)^4 \exp \left( C\frac{\log q}{\log\log q}\right)
\end{align*}
for an absolute constant $C>0$, where $\tau_k(n)=\sum_{m_1\dots m_k=n} 1$ denotes the $k$-fold divisor function. Resetting the constant, we have 
\begin{align}
\begin{split}
    \label{error-2}
    (\log X)^2 \sum_{\substack{d \mid q \\ d>1}} \sum_{l \mid d} \frac{\varphi(l)}{l} & \ll \exp \left( C\frac{\log q}{\log\log q}\right).
\end{split}
\end{align}

Therefore, by combining (\ref{S_0(X)}), (\ref{integral-|t|>T}), (\ref{integral-v}), (\ref{integral-h}), (\ref{error-1}) and (\ref{error-2}), we obtain
\begin{align*}
    \sum_{\substack{\chi\neq \chi_0 \\ \chi(-1)=1}} \frac{1}{\abs{L(1,\chi)}^2} &= \frac{\zeta(2)}{2\zeta(4)}\prod_{p \mid q} \left( 1-\frac{1}{p^2}\right)^{-1} \sum_{\substack{d \mid q \\ d>1}} \varphi^*(d) +O\left( \exp \left( C\frac{\log q}{\log\log q}\right) \right) \\
    &\quad+O\left( \delta_1\left(1-\beta_1\right)^{-1}\left( (\log q)^2+ (1-\beta_1)^{-1}\right)\right).
\end{align*}
Since $\sum_{\substack{d \mid q \\ d>1}} \varphi^*(d)=\varphi(q)-1$, we find that the main term in the above is 
\begin{align*}
    \frac{\zeta(2)}{2\zeta(4)}\prod_{p \mid q} \left( 1-\frac{1}{p^2}\right)^{-1} \sum_{\substack{d \mid q \\ d>1}} \varphi^*(d) &= \frac{\zeta(2)}{2\zeta(4)}\prod_{p \mid q} \left( 1-\frac{1}{p^2}\right)^{-1} \varphi(q) + O\left( \exp\left(\omega(q)\right)\right),
\end{align*}
where $\omega(n)$ denotes the number of distinct prime divisor of $n$. By using the estimate $\omega(q) \ll \frac{\log q}{\log\log q}$ (see \cite[Theorem 2.10]{MV}), we complete the proof of Theorem \ref{thm:negative square moment}.

\section{Proof of Theorem \ref{thm:negative square moment with conductor}}
We assume that $q=p^k$ is a prime power and that the exceptional zero does not exist throughout this section. We put $\mathfrak{b}=1$ in Lemma \ref{lem:S(X)} to obtain 
    \begin{align}
    \begin{split}
    \label{S_1(X)}
    S_1(X) = &\frac{\zeta(2)}{2\zeta(4)}\prod_{p \mid q} \left( 1-\frac{1}{p^2}\right)^{-1} \sum_{\substack{d \mid q \\ d>1}} \frac{\varphi^*(d)}{d} \\
    &\quad+O\left( X^{-\frac{1}{2}} \sum_{\substack{d \mid q \\ d>1}} \frac{1}{d}\sum_{l \mid d} \varphi(l) \right)+O\left((\log X)^2 \sum_{\substack{d \mid q \\ d>1}} \frac{1}{d}\sum_{l \mid d} \frac{\varphi(l)}{l}\right).
    \end{split}
    \end{align}

From the assumption, the function
\[
\tilde{g_\chi}(s):= \frac{X^{s-1}\Gamma(s-1)}{f_\chi L(s,\chi)L(s,\overline{\chi})}
\]
has only a simple pole at $s=1$ in the closed contour $\mathcal{C}_T$ which was defined in the previous section. From the residue theorem, we have
\begin{align*} 
    S_1(X) &= \sum_{\substack{\chi\neq \chi_0 \\ \chi(-1)=1}} \frac{1}{f_{\chi}\abs{L(1,\chi)}^2} \\
    &\quad + \frac{1}{2\pi i} \left(\int_{3+iT}^{A(c_1)+iT}+ \int_{A(c_1)+iT}^{A(c_1)-iT}+\int_{A(c_1)-iT}^{3-iT}\right)G_1(s)X^{s-1}\Gamma(s-1)ds \\
    &\quad +\frac{1}{2\pi i} \int_{\substack{\sigma=3 \\\abs{t} >T}} G_1(s)X^{s-1}\Gamma(s-1)ds.
\end{align*}

Putting $T=q, X =\exp\left(\frac{6}{c_0} \log q \log\log q \right)$, and by the same argument as in the proof of Theorem \ref{thm:negative square moment}, we find that 
\begin{align}
\begin{split}
\label{integral-|t|>T with conductor}
    \frac{1}{2\pi i} \int_{\substack{\sigma=3 \\\abs{t} >T}} G_1(s)X^{s-1}\Gamma(s-1)ds &\ll \sum_{\substack{\chi \neq \chi_0 \\ \chi(-1)=1}} \frac{1}{f_\chi} X^2 \int_T^\infty \abs{\Gamma(2+it)}dt \\ 
    &\ll X^2 q^\frac{3}{2}e^{-\frac{\pi}{2}q} \sum_{\substack{\chi \neq \chi_0 \\ \chi(-1)=1}} \frac{1}{f_\chi} \\ 
    &\ll k\exp\left(-\frac{\pi}{2}q\left(1+O\left(\frac{\log q \log\log q}{q} \right)\right)\right) 
\end{split}
\end{align}
since 
\[
\sum_{\substack{\chi \neq \chi_0 \\ \chi(-1)=1}} \frac{1}{f_\chi} \leq \frac{k}{2}.
\]
By the same argument as above, we have
\begin{align}
\begin{split}
\label{integral-v with conductor}
    \frac{1}{2\pi i} \int_{A(c_1)+iT}^{A(c_1)-iT}G_1(s)X^{s-1}\Gamma(s-1)ds &\ll \sum_{\substack{\chi \neq \chi_0 \\ \chi(-1)=1}} \frac{1}{f_\chi} (\log q)^2\exp \left( -\frac{c_0/2}{\log q}\log X\right) \\
    &\ll \frac{k}{\log q}
\end{split}
\end{align}
and
\begin{align}
\begin{split}
\label{integral-h with conductor}
    \int_{A(c_1) \pm iT}^{3 \pm iT}G_1(s)X^{s-1}\Gamma(s-1)ds 
    &\ll \sum_{\substack{\chi \neq \chi_0 \\ \chi(-1)=1}}\frac{1}{f_\chi} (\log q)^2 \frac{X^{2} q^{\frac{3}{2}} e^{-\frac{\pi q}{2}}}{\log (qX)} \\
    &\ll k \exp \left( -\frac{\pi}{2}q\left(1+O\left(\frac{\log q \log\log q}{q}\right)\right)\right).
\end{split}
\end{align}
Since $q=p^k$ is a prime power, (\ref{integral-|t|>T with conductor}), (\ref{integral-v with conductor}) and (\ref{integral-h with conductor}) can be estimated by 
\begin{align}
\label{error-integral}
\ll \frac{k}{\log q} = \frac{1}{\log p}.
\end{align}

In order to complete the proof, we estimate the error terms in (\ref{S_1(X)}). Since $\varphi(n)=n\prod_{p \mid n} \left(1-1/p\right)$, we have
\begin{align*}
(\log X)^2 \sum_{\substack{d>1 \\ d \mid q}} \frac{1}{d}\sum_{l \mid d} \frac{\varphi(l)}{l} &\ll (\log q \log\log q)^2 \sum_{j=1}^k \frac{j+1}{p^j} \left(1-\frac{1}{p}\right).
\end{align*}
We now invoke the generating function for the sum of $l$-th powers
\[
\sum_{l=1}^\infty l^n x^l = \frac{x}{(1-x)^{n+1}}A_n(x),
\]
where $A_n(x)$ are the Eulerian polynomials which are defined by the exponential generating function
\[
\sum_{n=0}^\infty A_n(t) \frac{x^n}{n!} = \frac{t-1}{t-e^{(t-1)x}}.
\]
By using the fact $A_1(x)=1$ and the definition of $X$, we have
\begin{align*}
\sum_{j=1}^k \frac{j+1}{p^j}\left(1-\frac{1}{p}\right) \ll \sum_{j=1}^\infty \frac{j}{p^j}+\sum_{j=1}^\infty \frac{1}{p^j} \ll \frac{1}{p}.
\end{align*}
Hence we have
\begin{align*}
(\log q \log\log q)^2 \sum_{j=1}^k \frac{j+1}{p^j}\left(1-\frac{1}{p}\right) &\ll \frac{(\log q \log\log q)^2 }{p} \\
&\ll \frac{k^2(\log p)^2(\log k+\log\log p)^2}{p}.
\end{align*}
So we obtain
\begin{align}
\label{p^k error 2}
(\log X)^2 \sum_{\substack{d>1 \\ d \mid q}} \frac{1}{d}\sum_{l \mid d} \frac{\varphi(l)}{l} &\ll \frac{k^2(\log p)^2(\log k+\log\log p)^2}{p}.
\end{align}
Since $n = \sum_{d \mid n} \varphi(d)$, the second term in (\ref{S_1(X)}) is 
\begin{align}
\label{p^k error 1}
X^{-\frac{1}{2}} \sum_{\substack{d>1 \\ d \mid q}} \frac{1}{d}\sum_{l \mid d} \varphi(l) &= k \exp \left(-\frac{3}{c_0}\log q\log\log q\right).
\end{align}
Finally, since $\varphi^*(p^j)=p^{j-2}(p-1)^2$, the contribution of the first term is
\begin{align}
\label{p^k main}
\frac{\zeta(2)}{2\zeta(4)}\left( 1+\frac{1}{p^2}\right)^{-1} \sum_{\substack{d>1 \\ d \mid q}} \frac{\varphi^*(d)}{d} &= 
\frac{\zeta(2)}{2\zeta(4)} \frac{(p-1)^2}{p^2+1} k.
\end{align}
Therefore, combining (\ref{S_1(X)}), (\ref{error-integral}), (\ref{p^k error 2}), (\ref{p^k error 1})and (\ref{p^k main}), we obtain
\begin{align}
\label{S(X) p^k}
S(X)= \frac{\zeta(2)}{2\zeta(4)} \frac{(p-1)^2}{p^2+1} k + O\left(\frac{k^2(\log p)^2(\log k+\log\log p)^2}{p}\right).
\end{align}
Therefore we complete the proof of Theorem \ref{thm:negative square moment with conductor}.

\section{Application in recovering short generators}

In this section, we describe an application of our main results to cryptography.
We consider the short generator problem, for more details, see~\cite{CDPR15, HWB17}.

\begin{defin}[The short generator problem]
    Let $K$ be a number field with $\mathcal{O}_K$ its ring of integers.
    Given a generator $h$ of the principal ideal $h\mathcal{O}_K$, the goal of the Short Generator Problem is to recover a generator $g$ of it satisfying $\norm{\Log{g^\prime}}_2=\min_{u\in\mathcal{O}_K^*}\norm{\Log{g\cdot u}}_2$.
    Such a generator $g$ is called a shortest generator of the principal ideal $h\mathcal{O}_K$.
\end{defin}

Notice that the notation $\Log{\cdot}$ in the above definition is the logarithmic embedding of the number field defined for a $q$-th number field $K$ as $\Log{\alpha}=\left(\log\abs{\sigma_i (\alpha)}\right)_{i\in G}\in\mathbb{R}^{\varphi (q)/2}$ for all $\alpha\in K$, where $\sigma_j$ are the complex embeddings and $G\coloneqq\left(\mathbb{Z}_q^*/\{\pm 1\}\right)$.
We consider $K=\Q (\zeta_q )$ to be a $q$-th cyclotomic number field with its group of cyclotomic units $\CycUnitGp =\langle -1,\zeta_q , b_j\mid j\in G\setminus\{ 1\}\rangle$ with $b_j=\frac{\zeta_q^j-1}{\zeta_q -1}$.
Let $\bb_j=\Log{b_j}=\Log{b_{-j}}$ for $j\in G\setminus\{ 1\}$.
Then $\{\bb_j\}$ forms a basis for $\Log{\mathcal{C}}$, the log-cyclotomic-unit lattice of $K$.
We denote by $\{\bb_j^\vee\}$ its dual basis corresponding to $\{\bb_j\}$.

\subsection{Previous results on SGP over cyclotomic number fields}

We combine the implication of Theorem~4.1 of~\cite{CDPR15}
to the special case using the distribution given by Lemma~5.4 of~\cite{CDPR15}, which is the main target of our application, more details and proofs is given in Appendix~\ref{App:SGP}.

\begin{thm}[{implication of \cite[Theorem 3.1, Theorem 4.1, Lemma 5.4]{CDPR15}}]
\label{thm:CDPR_main}
    Let $X_1,\dots ,$ $X_{\varphi (q)/2}, X^\prime_1,\dots , X^\prime_{\varphi (q)/2}$ be i.i.d. random variables of the Gaussian distribution with the mean 0 and the standard deviation $r>0$, and let $\hat{X}_i=(X_i^2+X_i^{\prime 2})^{1/2}$.
    Then for any tuple of vectors $\vv^{(1)},\dots ,\vv^{(\varphi (q)/2-1)}\in\R^{\varphi (q)/2}$ of Euclidean norm 1 that are orthogonal to the all-1 vector, and for parameter $t$ such that
    $\frac{1}{2\norm{\bb_j^\vee}_2}>t>T$ for some universal constant $T$,
    \[\prob{\exists j,\abs{\sum_ia_i^{(j)}\log (\hat{X}_i)}\geq t}\leq \left(\varphi (q)-2\right) e^{-t/2}.\]
    Then there is an efficient algorithm that given $g^\prime =g\cdot u$, where $g$ is chosen from the distribution given by $\left(\hat{X}_i\right)$ and $u$ is a cyclotomic unit, outputs an element of the form $\pm\zeta^jg$ with probability at least $1-\left(\varphi (q)-2\right)e^{(-t/2)}$.
\end{thm}

In~\cite{CDPR15}, the relation between the length of the dual basis and the negative square moment of $L\left( 1,\chi\right)$ is also given.
We rephrase the combination of Theorem~3.1 and the equation given in its proof.

\begin{thm}[{\cite[from the proof of Theorem 3.1]{CDPR15}}]
\label{thm:313b}
    Let $q=p^k$ for a prime $p$, and let $\{\bb^\vee_j\}_{j\in G\setminus\{ 1\}}$ denote the basis dual to $\{\bb_j\}_{j\in G\setminus\{ 1\}}$.
    Then all $\norm{\bb^\vee_j}_2$ are equal, and
    \begin{align*}
        \lVert\bb_j^\vee\rVert_2^2 & = 4\abs{G}^{-1}\cdot\sum_{\chi\in\hat{G}\setminus\{ 1\}}f_{\chi}^{-1}\cdot\abs{L\left( 1,\chi\right)}^{-2},
    \end{align*}
    where $\hat{G}$ is the set of characters of $G$.
\end{thm}

Here we mark out the estimate of the length of the dual basis given by~\cite{CDPR15}.

\begin{thm}[{\cite[part of Theorem~3.1]{CDPR15}}]
\label{thm:CDPR_b_uncond}
    Let $q=p^k$ for a prime $p$, and let $\{\bb^\vee_j\}_{j\in G\setminus\{ 1\}}$ denote the basis dual to $\{\bb_j\}_{j\in G\setminus\{ 1\}}$.
    Then all $\norm{\bb^\vee_j}_2$ are equal, and
    \[\lVert\bb_j^\vee\rVert_2^2\leq 2k\abs{G}^{-1}\cdot\left(\ell (q)^2+O(1)\right) =4C^2 k\frac{\left(\log q\right)^2}{q}(1+o(1)),\]
    where $\ell (q)=C\log q$ for some $C>0$.
\end{thm}

\subsection{Our improvements}

In this subsection, we show the improvement on estimating the length of the dual basis obtained from applying our main theorems.

If $q$ is a large prime number, then $f_\chi =q$ for all $\chi$ modulo $q$ and $\chi\neq \chi_0$. Hence by substituting (\ref{consequence of the Siegel Theorem}) we have
\begin{align*}
    \lVert \bb_j^\vee \rVert_2 = \sqrt{\frac{4\zeta(2)}{\zeta(4)q} \left(1+O\left(q^{-1+\eps}\right)\right)} = \frac{2\sqrt{15}}{\pi \sqrt{q}} \left(1+O\left(q^{-1+\eps}\right)\right).
\end{align*}
Then there exists $Q>0$ such that for any $q>Q$,
    \[
    \lVert \bb_j^\vee \rVert_2 <\frac{2\sqrt{15}}{\pi \sqrt{q}} \left(1+\delta\right)
    \]
    for some positive $\delta\ll 1$.
    Then by applying Theorem~\ref{thm:CDPR_main}, we obtain the following result.

\begin{cor}
\label{cor:algo}
    Let $q$ be a large prime number.
    There exists an efficient algorithm that given $g^\prime =g\cdot u$, where $g$ is chosen from $D(t,\alpha)$ and $u\in\CycUnitGp$ is a cyclotomic unit, outputs an element of the form $\pm\zeta^jg$ with probability at least $\alpha =1-(q-3)e^{-t/2}$ with $t=\frac{\pi}{4\sqrt{15}}\frac{\sqrt{q}}{1+\delta}>T$ for some small $\delta >0$.
\end{cor}

For the case $q=p^k$ is a prime power, first we assume that $k =o\left(p/(\log p)^4\right)$ with $p$ large enough. Under the assumption of the nonexistence of the exceptional zero, we obtain
\begin{align*}
    \lVert \bb_j^\vee \rVert_2 &= \sqrt{\frac{8}{\varphi(p^k)}\left(\frac{\zeta(2)}{2\zeta(4)} \frac{(p-1)^2}{p^2+1} k\left(1
    +o(1)\right)\right)} \\
    &= \frac{2\sqrt{15}}{\pi} \sqrt{\frac{k}{\varphi(p^k)}}\left(1
    +o(1)\right).
\end{align*}
While for the case $k \gg p/(\log p)^4$, under the assumption of the nonexistence of the exceptional zero, we have 
\begin{align*}
    \lVert \bb_j^\vee \rVert_2 &\ll \frac{k \log p(\log k+\log\log p)}{\sqrt{\varphi(p^k)p}} .
\end{align*}
Then there exists $Q_1>0$ such that for any $q>Q_1$,
    \begin{align*}
        \lVert \bb_j^\vee \rVert_2 & = \frac{2\sqrt{15}}{\pi} \sqrt{\frac{k}{\varphi(p^k)}}\left(1 +o(1)\right)\\
        & <\frac{2\sqrt{15}}{\pi} \sqrt{\frac{k}{\varphi(p^k)}}\left(1 +\delta_1\right)
    \end{align*}
    for some positive $\delta_1\ll 1$.
    Otherwise, there exists $Q_2>0$ such that for any $q>Q_2$,
    \[
    \lVert \bb_j^\vee \rVert_2 =\bigO{\frac{k \log p(\log k+\log\log p)}{\sqrt{\varphi(p^k)p}}}\leq C_2\frac{k \log p(\log k+\log\log p)}{\sqrt{\varphi(p^k)p}}
    \]
    for some constant $C_2 >0$.
    Thus we derive the following result by applying Theorem~\ref{thm:CDPR_main}.

\begin{cor}
\label{cor:algo_pk}
    Let $q=p^k$ be a power of a prime number and assume that the exceptional zeros do not exist.
    There exists an efficient algorithm that given $g^\prime =g\cdot u$, where $g$ is chosen from $D(t,\alpha)$ and $u\in\CycUnitGp$ is a cyclotomic unit, outputs an element of the form $\pm\zeta^jg$ with probability at least $\alpha =1-(\varphi (q)-2)e^{-t/2}$ with $t=\frac{\pi}{4\sqrt{15}}\sqrt{\frac{\varphi (q)}{k}}\frac{1}{1+\delta}>T$ for some small $\delta >0$ if $k=o\left(\frac{p}{\left(\log p\right)^4}\right)$ and $T<t\leq\frac{\sqrt{\varphi (q)p}}{2\log q\log\log q}$ otherwise.
\end{cor}

We list the estimates on the length of the dual basis and the corresponding lower bounds on the success probability of the algorithm described in Theorem~\ref{thm:CDPR_main}, in terms of $t$, in Table~\ref{table}.
One can observe that when $q$ is a prime number, we improve $t$ by a log term and also get the optimal description of order of $t$ for Theorem~\ref{thm:CDPR_main}'s setting other than only a lower bound.
Notice that compared to the prior work~\cite{NT24+}, we remove the GRH assumption.
For prime power $q$, the bound for $L(1,\chi)$ doesn't change, so that~\cite{CDPR15} has the same estimate for the dual basis.
Therefore, we improve the estimate under assuming the exceptional zeros don't exist for both cases, $k =o\left(p/(\log p)^4\right)$ with large prime $p$ (``first $q^k$ condition" henceforth), and $k \gg p/(\log p)^4$.
Moreover, since the estimates on the dual basis for prime $q$ and for the first $q^k$ condition are equalities, their parameters $t$ and the corresponding lower bound on the success probability are optimal for Theorem~\ref{thm:CDPR_main}'s setting.

\begin{center}
\begin{table}[H]
\caption{Comparison: the constant $C$ refer to Theorem~\ref{thm:CDPR_b_uncond}.}
\label{table}
\begin{tabular}{|l|c|c|}
\hline
 Condition and approach & $\displaystyle\norm{\bb^\vee_j}_2$ & Parameter $t$ in Theorem~\ref{thm:CDPR_main}\\
 \hline
 $q=p^k$; bound on $L(1,\chi)$~\cite{CDPR15} & $\displaystyle \leq 2C\sqrt{k}\frac{\log q}{\sqrt{q}}(1+o(1))$
 & $\displaystyle \gg \frac{\sqrt{q}}{\sqrt{k}\log q}$ \\
 \hline
 \multirow{2}{14em}{$q$: large prime; negative moment} & & \\
 & $\displaystyle =\frac{2\sqrt{15}}{\pi \sqrt{q}} \left(1+O\left(q^{-1+\eps}\right)\right)$ & $\displaystyle \asymp\sqrt{q}$ \\
 \hline
 \multirow{5}{14em}{$\displaystyle q=p^k$ with $p$: large prime, $k =o\left(p/(\log p)^4\right)$; negative moment (assuming the nonexistence of exceptional zeros)} & & \\
  & &  \\
  & $\displaystyle = \frac{2\sqrt{15}}{\pi} \sqrt{\frac{k}{\varphi(q)}}\left(1 + o(1)\right)$
  & $\displaystyle \asymp \sqrt{\frac{q}{k}}$
  \\
  & & \\
  & & \\
 \hline
 \multirow{5}{14em}{$\displaystyle q=p^k$ with $p$: prime, $k \gg p/(\log p)^4$; negative moment (assuming the nonexistence of exceptional zeros)} & & \\
 & & \\
  & $\displaystyle \ll \frac{\log q\log\log q}{\sqrt{\varphi(q)p}}$
  & $\displaystyle \gg \frac{\sqrt{qp}}{\log q\log\log q}$
  \\
  & & \\
   & & \\
\hline
\end{tabular}
\end{table}
\end{center}

\begin{ack} 
The authors would like to thank Professor Kohji Matsumoto for reading the manuscript carefully and giving them valuable comments. The authors would also like to thank Professor Fran{\c{c}}ois Le Gall for his financial supports. I.-I. Ng is supported by MEXT Q-LEAP grant No. JPMXS0120319794.
Y. Toma is supported by Grant-in-Aid for JSPS Research Fellow grant No. 24KJ1235. 
\end{ack} 

\printbibliography

\appendix
\section{Proof Theorem~\ref{thm:CDPR_main}}
\label{App:SGP}

In this section, we describe the details and the proof of Theorem~\ref{thm:CDPR_main}.
As in~\cite{CDPR15}, we assume that the index $[\Log{\mathcal{O}_K^*}:\Log{\CycUnitGp}]$ is small.

We denote by $D(t,\alpha)$ the distribution over $K$ with the property that for any tuple of vectors $\vv_1,\dots ,\vv_{\ell}\in\R^{\varphi (q)/2}$ of Euclidean norm 1 that are orthogonal to the all-1 vector $\boldsymbol{1}$, the probability that $\abs{\langle\Log{g},\mathbf{v}_j\rangle}<t$ holds for all $j=1,\dots ,\ell$ is at least some $\alpha >0$, where $g$ is chosen from $D(t,\alpha)$ and the parameter $t$ is positive (remembering that $\Log{K}\subset\R^{\varphi (q)/2}$).

The following is an immediate result from the fact that the image $\Log{u}\in\R^{\varphi (q)/2}$ for $u\in\Ok^*$ is orthogonal to the all-1 vector.

\begin{lem}
\label{lem:districution_Log(u)}
    Let $g$ be chosen from $D(t,\alpha)$.
    Then $\abs{\left\langle\Log{g},\frac{\bb_j^{\vee}}{\norm{\bb_j^{\vee}}_2}\right\rangle}<t$ holds for all $j=1,\dots ,\varphi (q)/2$ with probability at least $\alpha$.
\end{lem}

Below, we specify the parameters in our setting of $K$ for some certain Gaussian distributions.

\begin{lem}
    \label{lem:prob}
    Let $n=\varphi (q)/2$ and $\ell =n-1$.
    Then $D(t,\alpha)$ with $t>0$ and $\alpha =1-(\varphi (q)-2)e^{-t/2}$ exists for some Gaussian distributions that have standard deviation $r$ if $t\geq T$, where $T$ is the universal constant given in~\cite[Lemma 5.4]{CDPR15}.
\end{lem}

\begin{proof}[Proof of Lemma~\ref{lem:prob}]
    The corollary follows immediately from~\cite[Lemma 5.4]{CDPR15} over $K$ by taking $n=\varphi (q)/2$ and $\ell =n-1$.
    Here we identify the elements of $K$ by real and imaginary parts of their image under complex embeddings, i.e., $\Psi :\Ok\rightarrow\R^{\varphi (q)}$ such that $\Psi (a)=(\operatorname{Re}(\sigma_j(a)),\operatorname{Im}(\sigma_j(a)))_{j=1,\dots ,\varphi (q)/2}$.
    More precisely, the random variables $X_i$ and $X^{\prime}_i$ correspond to $\left(\operatorname{Re}(\sigma_i (a))\right)_{a\in K}$ and $\left(\operatorname{Im}(\sigma_i(a))\right)_{a\in K}$, respectively.
    It follows that the random variables $\hat{X}_i$ correspond to $\left(\abs{\sigma_i (a)}\right)_{a\in K}$ and thus $\log\left(\hat{X}_i\right)$ correspond to  $\left(\Log{a}\right)_{a\in K}$.
    Let $\vv^{(1)} ,\dots ,\vv^{(\varphi (q)/2-1)}\in\R^{\varphi (q)/2}$ be vectors with Euclidean norm 1 that are orthogonal to the all-1 vector.
    Therefore,~\cite[Lemma 5.4]{CDPR15} indicates that for entries $v^{(j)}_i$ of $\vv^{(j)}$,
    \[
    \prob{\exists j,\abs{\sum_i v_i^{(j)}\log (\hat{X}_i)}\geq t}\leq \left(\varphi (q)-2\right) e^{-t/2},
    \]
    which implies that
    \begin{align*}
    & \prob{\abs{\left\langle\left(\Log{\hat{X}_i}\right)_{i=1,\dots ,\varphi (q)/2},\vv^{(j)}\right\rangle} <t\text{ for all }j=1,\dots ,\varphi (q)/2-1}\\
    & >1-\left(\varphi (q)-2\right) e^{-t/2}.
    \end{align*}
    Hence, $\hat{X}_1 ,\dots ,\hat{X}_{\varphi (q)/2}$ are i.i.d. $D\left( t, 1-\left(\varphi (q)-2\right) e^{-t/2}\right)$ random variables.
\end{proof}

We first show an immediate consequence deduced from results of~\cite{CDPR15}.
\begin{cor}
\label{cor:CDPR_main}
    If the parameter $t$ satisfies $\frac{1}{2\norm{\bb_j^\vee}_2}>t>T$ for some universal constant $T$, then there is an efficient algorithm that given $g^\prime =g\cdot u$, where $g$ is chosen from the distribution given by $D(t,\alpha)=D(t, 1-\left(\varphi (q)-2\right)e^{-t/2})$ and $u\in\CycUnitGp$ is a cyclotomic unit, outputs an element of the form $\pm\zeta^jg$ with probability at least $\alpha =1-\left(\varphi (q)-2\right)e^{-t/2}$.
\end{cor}

We mainly follow the proof of~\cite[Theorem 4.1]{CDPR15}, and take care of the parameters and probability.

\begin{proof}
Let $t$ satisfies that $\frac{1}{2\norm{\bb_j^{\vee}}_2}>t>T$.
The idea is to find the magnitude of $u$ by computing $\Log{u}$, and to divide $g^{\prime}$ by $u$.
Notice that under the logarithmic embedding, we have the relation $\Log{g^{\prime}}=\Log{g}+\Log{u}\in\R^{\varphi (q)/2}$ with $\Log{u}\in\Log{\CycUnitGp}$ and $\Log{g}\in\Log{\Ok^*}$, whose form is suitable for Babai's round-off algorithm~(\cite[Claim 2.1]{CDPR15}).

Then the algorithm goes by first finding $\Log{u}$ with Babai's algorithm, then computing $u^{\prime}=\prod b_j^{a_j}$, where $a_j$ are integer coefficients of $\Log{u}=\sum a_j\bb_j$, and finally outputting $g^{\prime}/u^{\prime}$.
Since $\Log{u^{\prime}}=\Log{u}$ implies that $g^{\prime}/u^{\prime}=gu/u^{\prime}=\pm\zeta^jg$ for some sign and some $j\in\{1,\dots , q\}$, it suffices to show the fitness for applying Babai's algorithm, and the probability for allowing to apply it.

From Theorem~\ref{thm:313b}, the short generator algorithm succeeds only if Babai's algorithm succeeds.
According to~\cite[Claim 2.1]{CDPR15}, in order to apply Babai's algorithm, we need to ensure the input generator $g$ meets the requirement that 
\[\abs{\langle\Log{g},\bb_j^{\vee}\rangle}<1/2\]
for all $j\in G\setminus\{ 1\}$.

By Lemma~\ref{lem:districution_Log(u)} and the assumption $\frac{1}{2\norm{\bb_j^{\vee}}_2}>t$, we have 
\begin{align*}
    \abs{\langle\Log{g},\bb_j^{\vee}\rangle} & =\lVert\bb_j^{\vee}\rVert_2\cdot\abs{\left\langle\Log{g},\frac{\bb_j^{\vee}}{\lVert\bb_j^{\vee}\rVert_2}\right\rangle}\\
    & <\lVert\bb_j^{\vee}\rVert_2\cdot t\\
    & <\frac{1}{2}
\end{align*}
as claimed.
Then by Lemma~\ref{lem:prob} and the assumption that $t\geq T$, the success probability is lower bounded by $\alpha =1-(\varphi (q)-2)e^{-t/2}$ as claimed. 
\end{proof}

Then Theorem~\ref{thm:CDPR_main} immediately follows~\cite[Theorem 3.1, Theorem 4.1, Lemma 5.4]{CDPR15} and Corollary~\ref{cor:CDPR_main}.

\end{document}